\documentclass{article}

\usepackage{graphicx,youngtab}
\usepackage{amsmath,amsfonts,amsthm}

\newtheorem{theorem}{Theorem}[section]
\newtheorem{lemma}[theorem]{Lemma}

\newtheorem{corollary}[theorem]{Corollary}

\theoremstyle{definition}

\newtheorem{example}[theorem]{Example}

\theoremstyle{remark}
\newtheorem{remark}[theorem]{Remark}

\numberwithin{equation}{section}

\DeclareMathOperator{\nbe}{n_b}
\DeclareMathOperator{\RealPart}{Re}

\begin{document}

\title{Wallis-Ramanujan-Schur-Feynman}
\author{T.~Amdeberhan, O.~R.~Espinosa, V.~H.~Moll and A.~Straub}
\date{}
\maketitle

\begin{abstract}
One of the earliest examples of analytic representations for $\pi$ 
is given by an infinite product
provided by Wallis in $1655$. The modern literature often presents
this evaluation based on the integral formula
$$ \frac{2}{\pi} \int_0^\infty \frac{dx}{(x^2+1)^{n+1}} = \frac{1}{2^{2n}} \binom{2n}{n}. $$
In trying to understand the behavior of this integral when the integrand
is replaced by the inverse of a 
product of distinct quadratic factors, the authors encounter relations
to some formulas of Ramanujan, expressions involving Schur functions,
and Matsubara sums that have appeared in the context of Feynman 
diagrams. 
\end{abstract}

\section{Wallis's infinite product for  $\pi$}\label{sec-intro}

Among the earliest analytic expressions for $\pi$ one finds two infinite
products: the first one given by Vieta \cite{vieta1} in $1593$,
\[
  \frac{2}{\pi} = \sqrt{\frac{1}{2}}
  \sqrt{\frac{1}{2} + \frac{1}{2} \sqrt{\frac{1}{2}}}
  \sqrt{\frac{1}{2} + \frac{1}{2} \sqrt{\frac{1}{2} + \frac{1}{2} \sqrt{\frac{1}{2}}} } \cdots,
\]
and the second by Wallis \cite{wallis1} in $1655$,
\begin{equation}\label{wallis-0}
\frac{2}{\pi} = \frac{1 \cdot 3}{2 \cdot 2} \cdot
\frac{3 \cdot 5}{4 \cdot 4} \cdot
\frac{5 \cdot 7}{6 \cdot 6}  \cdot
\frac{7 \cdot 9}{8 \cdot 8} \cdots.
\end{equation}
In this journal, T. Osler \cite{osler1}
has presented the remarkable formula
\[
  \frac{2}{\pi} =\prod_{n=1}^{p}
  \sqrt{\frac{1}{2} + \frac{1}{2} \sqrt{\frac{1}{2} +
  \sqrt{\frac{1}{2} + \cdots + \frac{1}{2} \sqrt{\frac{1}{2}} }}}
  \prod_{n=1}^{\infty} \frac{2^{p+1}n-1}{2^{p+1}n}
  \cdot \frac{2^{p+1}n+1}{2^{p+1}n},
\]
where the $n$th term in the first product has $n$ radical signs.
This equation becomes Wallis's product when $p=0$ and Vieta's formula as
$p \to \infty$. It is surprising that such a connection between the two products
was not discovered earlier.

The collection \cite{pi-formulas} contains both original papers of Vieta and
Wallis as well as  other fundamental
papers in the history of $\pi$. Indeed, there are many good historical sources on
$\pi$. The  text by P. Eymard and J. P. Lafon \cite{eymard-lafon} is an
excellent place to start.

Wallis's formula (\ref{wallis-0}) is equivalent to
\begin{equation}\label{wallis2}
W_n := \prod_{k=1}^{n} \frac{(2k) \cdot (2k)}{(2k-1) \cdot (2k+1)} =
\frac{2^{4n}}{\binom{2n}{n} \, \binom{2n+1}{n} \, (n+1) } \to \frac{\pi}{2}
\end{equation}
as $n \to \infty$. This may be established using Stirling's approximation
\[ m! \sim  \sqrt{2 \pi m} \left( \frac{m}{e} \right)^{m}. \]
Alternatively, there are many elementary proofs of (\ref{wallis2}) in the
literature. Among them, \cite{wastlund} and \cite{levrie-daems} have
recently appeared in this journal.

Section \ref{sec-wallis-pro} presents a proof of (\ref{wallis2}) based on
the evaluation of the rational integral
\begin{equation}\label{wallis1}
G_n := \frac{2}{\pi} \int_{0}^{\infty} \frac{dx}{(x^{2}+1)^n}.
\end{equation}
This integral is discussed in the next section.
The motivation to generalize (\ref{wallis1}) has produced interesting links
to symmetric functions from combinatorics and to
one-loop Feynman diagrams from particle physics.
The goal of this work is to present these connections.

\section{A rational integral and its trigonometric version}\label{sec-ratio}

The method of partial fractions reduces the integration of a rational function
to an algebraic problem: the factorization of its denominator. The integral 
(\ref{wallis1}) corresponds to the presence of purely imaginary poles. See
\cite{irrbook} for a treatment of these ideas. 

A recurrence for $G_n$ is
obtained by writing $1 = (x^{2}+1)-x^{2}$ for the numerator of (\ref{wallis1})
and integrating by parts. The result is
\begin{equation}\label{recur-Q}
G_{n+1} = \frac{2n-1}{2n} G_n.
\end{equation}
Since $G_1=1$ it follows that
\begin{equation}\label{qone}
G_{n+1} = \frac{1}{2^{2n}} \binom{2n}{n}.
\end{equation}

The choice of a new variable is one of the fundamental tools
in the evaluation of definite integrals. The new variable, if
carefully chosen, usually simplifies the problem or opens up
unsuspected possibilities. Trigonometric changes of variables are considered {\em elementary}
because these functions appear early in the scientific training. Unfortunately,
this hides the fact that this change of variables introduces a
{\em transcendental function} with a multivalued inverse. One has to
proceed with care.

The change of variables $x = \tan \theta$ in the definition (\ref{wallis1}) of $G_n$
gives
\[ G_{n+1} =  \frac{2}{\pi} \int_{0}^{\pi/2} ( \cos \theta)^{2n} \, d \theta. \]
In this context, the recurrence (\ref{recur-Q}) is obtained by writing
\[
  (\cos \theta)^{2n} = (\cos \theta )^{2n-2} +
  \frac{\sin \theta}{2n-1} \frac{d}{d \theta} (\cos \theta)^{2n-1}
\]
and then integrating by parts. Yet another recurrence for $G_n$ is obtained by a double-angle
substitution yielding
\[
  G_{n+1} = \frac{2}{\pi} \int_{0}^{\pi/2}
  \left( \frac{1+ \cos 2 \theta}{2} \right)^{n} \, d \theta,
\]
and a binomial expansion (observe that the odd powers of
cosine integrate to zero). It follows that
\[
  G_{n+1} = 2^{-n} \sum_{k=0}^{ \lfloor{n/2 }\rfloor} \binom{n}{2k} G_{k+1}.
\]
Thus, (\ref{qone}) is equivalent to proving the finite sum identity
\begin{equation}\label{finite-sum}
\sum_{k=0}^{\left\lfloor{n/2 }\right\rfloor}
2^{-2k} \binom{n}{2k} \binom{2k}{k}
= 2^{-n} \binom{2n}{n}.
\end{equation}
There are many possible ways to prove this identity. For instance, it
is a perfect candidate for the truly $21$st-century WZ method
\cite{aequalsb} that provides automatic proofs; or, as pointed out by M. Hirschhorn in \cite{hirschhorn2}, it
is a disguised form of the Chu-Vandermonde identity
\begin{equation}
\sum_{k \geq 0} \binom{x}{k} \binom{y}{k} = \binom{x+y}{x}
\label{chu-van}
\end{equation}
(which was discovered first in $1303$ by Zhu Shijie). Namely, upon employing Legendre's
duplication formula for the gamma function
\[ \Gamma( \tfrac{1}{2} ) \Gamma(2z+1) =
   2^{2z} \Gamma(z+1) \Gamma(z + \tfrac{1}{2} ), \]
the identity (\ref{finite-sum}) can be rewritten as
\[ \sum_{k \geq 0} \binom{\tfrac{n}{2}}{k}
   \binom{\tfrac{n}{2}- \tfrac{1}{2}}{k} =
   \binom{n- \tfrac{1}{2}}{\tfrac{n}{2}- \tfrac{1}{2}}. \]
This is a special case of (\ref{chu-van}). Another, particularly nice and direct,
proof of (\ref{finite-sum}), as kindly pointed out by one of the referees,
is obtained from looking at the constant coefficient of
\[ \left(\frac{x}{2}+\frac{x^{-1}}{2}+1\right)^n = 2^{-n} \left(x^{1/2}+x^{-1/2}\right)^{2n}. \]

\begin{remark}
The idea of double-angle reduction lies at the heart of the {\em rational
Landen transformations}.  These are polynomial maps on the coefficient of the
integral of a rational function that preserve its value. See \cite{manna-moll3}
for a survey on Landen transformations and open questions.
\end{remark}

\section{A squeezing method}\label{sec-wallis-pro}

In this section we employ the explicit expression for $G_n$
given in (\ref{qone}) to establish Wallis's formula (\ref{wallis-0}). This
approach is also contained in Stewart's calculus text book \cite{stewart1} in the form of
several guided exercises (45, 46, and 68 of Section 7.1). The
proof is based on analyzing the integrals
\[ I_n := \int_{0}^{\pi/2} (\sin x)^n \, dx. \]
The formula
\[ I_{2n} = \int_{0}^{\pi/2} (\sin x)^{2n} \, dx = \frac{(2n-1)!!}{(2n)!!} \frac{\pi}{2} \]
follows from (\ref{qone}) by symmetry. Its companion integral
\[ I_{2n+1} = \int_{0}^{\pi/2} (\sin x)^{2n+1} \, dx = \frac{(2n)!!}{(2n+1)!!} \]
is of the same flavor. Here
$n!! = n(n-2)(n-4) \cdots \, \{ 1 \text{ or } 2 \}$ denotes
the double factorial. The ratio of these two integrals gives
\[ W_n I_{2n}/I_{2n+1} = \frac{\pi}{2}, \]
where $W_n$ is defined by (\ref{wallis2}). The convergence
of $W_n$ to $\pi/2$ now follows from the inequalities
$1 \leq I_{2n}/I_{2n+1} \leq 1 + 1/(2n)$.  The first of these inequalities is
equivalent to $I_{2n+1} \leq I_{2n}$, which holds because $(\sin x)^{2n+1} \leq (\sin x)^{2n}$.
The second is equivalent to
\[
  2n \int_{0}^{\pi/2} (\sin x)^{2n} \, dx \leq (2n+1) \int_{0}^{\pi/2}
  (\sin x)^{2n+1}\, dx,
\]
which follows directly from the bound
$I_{2n} \leq I_{2n-1}$ and the recurrence $(2n+1) I_{2n+1} = 2n I_{2n-1}$.
Alternatively, the second inequality can be proven by observing that the function
\[ f(s) = s \int_{0}^{\pi/2} (\sin x)^{s}\, dx \]
is increasing. This may be seen from the change of variables $t = \sin x$ and
a series expansion of the new integrand yielding
\begin{equation}\label{derivative}
f'(s) = \sum_{k=0}^{\infty} \frac{1}{2^{2k}} \binom{2k}{k}
\frac{2k+1}{(2k+s + 1)^{2}} > 0.
\end{equation}

\begin{remark}
Comparing the series (\ref{derivative}) at $s=0$ with the limit
\[
  f'(0) = \lim\limits_{s \to 0} \frac{f(s)}{s} =
  \lim\limits_{s \to 0} \int_{0}^{\pi/2} \sin^{s} x \, dx = \frac{\pi}{2}
\]
immediately proves
\[ \sum_{k=0}^{\infty} \binom{2k}{k} \frac{2^{-2k}}{2k+1} = \frac{\pi}{2}. \]
This value may also be obtained by letting $x=\tfrac{1}{2}$ in the series
\[
  \sum_{k=0}^{\infty} \binom{2k}{k} \frac{x^{2k}}{2k+1} =
  \frac{\text{arcsin} 2x }{2x}.
\]
The reader will find in \cite{lehmer85} a host of other interesting series
that involve  the central binomial coefficients.
\end{remark}

\section{An example of Ramanujan and a generalization}\label{sec-rama}

A natural generalization of Wallis's integral (\ref{wallis1}) is given by
\begin{equation}\label{wallis-3}
G_{n}(\boldsymbol{q}) = \frac{2}{\pi} \int_0^\infty \prod_{k=1}^{n} \frac{1}{x^2+q_k^2} \, dx,
\end{equation}
where $\boldsymbol{q} = (q_1, q_2, \ldots, q_n)$ with $q_k \in \mathbb{C}$. This
notation will be employed throughout. Similarly, $\boldsymbol{q}^\alpha$ is used to denote
$(q_1^\alpha, q_2^\alpha, \ldots, q_n^\alpha)$. As the value of 
the integral (\ref{wallis-3}) is fixed under a change of 
sign of the parameters $q_k$, it is assumed that $\RealPart{q_k} > 0$.  Note that 
the integral $G_n(\boldsymbol{q})$ is a symmetric function of 
$\boldsymbol{q}$ that
reduces to $G_n$ in the special case $q_1 = \cdots = q_n = 1$.

The special case $n=4$ appears as Entry 13, Chapter 13 of volume 2 of B. Berndt's
{\it Ramanujan's Notebooks} \cite{berndtII}, in the form\footnote{A minor correction from \cite{berndtII}.}:

\begin{example}
Let $q_1$, $q_2$, $q_3$, and $q_4$ be positive real numbers. Then
\begin{multline}
\nonumber
\frac{2}{\pi} \int_{0}^{\infty} \frac{dx}{(x^{2} + q_{1}^{2})
(x^{2} + q_{2}^{2})(x^{2} + q_{3}^{2})
(x^{2} + q_{4}^{2})} = \\
\frac{( q_{1}+q_{2}+q_{3}+q_{4})^{3} -
(q_{1}^{3} + q_{2}^{3} + q_{3}^{3} + q_{4}^{3}) }
{3q_{1}q_{2}q_{3}q_{4} (q_{1}+ q_{2})
(q_{2}+q_{3})(q_{1}+q_{3})(q_{1}+q_{4})
(q_{2}+q_{4})(q_{3}+q_{4})}.
\end{multline}
\end{example}

Using partial fractions the following general formula for $G_{n}(\boldsymbol{q})$
is obtained. In the next section a representation in terms of Schur functions is
presented.

\begin{lemma}\label{Gnqpartialfractions}
Let $\boldsymbol{q} = (q_1, \ldots, q_n)$ be distinct and $\RealPart{q_k} > 0$. Then
\begin{equation}\label{formula-0}
G_{n}(\boldsymbol{q}) =
\sum_{k=1}^{n} \frac{1}{q_k}
 \mathop{\prod_{j=1}^n}_{j \neq k} \frac{1}{q_j^2-q_k^2}.
\end{equation}
\end{lemma}

\begin{proof}
Observe first that if $b_1,b_2, \ldots, b_n$ are distinct then
\begin{equation}
\prod_{k=1}^{n} \frac{1}{y + b_{k}} = \sum_{k=1}^{n} \frac{1}{y+b_{k}} \mathop{\prod_{j=1}^n}_{j \neq k} \frac{1}{b_{j}-b_{k}}.
\label{pf-1}
\end{equation}
Replacing $y$ by $x^{2}$ and $b_{k}$ by $q_{k}^{2}$ and using the
elementary integral
\[ \frac{2}{\pi} \int_{0}^{\infty} \frac{dx}{x^{2} + q^{2}} = \frac{1}{q} \]
produces the desired evaluation of $G_{n}(\boldsymbol{q})$.
\end{proof}

\begin{remark}
$G_{n}(\boldsymbol{q})$,  as defined by (\ref{wallis-3}), 
 is a symmetric function in the
$q_i$'s which remains finite if two of these parameters coincide. Therefore, the factors
$q_j-q_k$ in the denominator of the right-hand side of (\ref{formula-0}) cancel out.
This may be checked directly by combining the summands corresponding to $j$ and $k$.
Alternatively, note that the right-hand side of (\ref{formula-0}) is symmetric while the
critical factors $q_j-q_k$ in the denominator combine to form the antisymmetric Vandermonde
determinant. Accordingly, they have to cancel.
\end{remark}

\begin{example}\label{eg-specializations}
The identities
\begin{align*}
  \frac{2}{\pi} \int_{0}^{\infty} \prod_{j=1}^{n+1} \frac{1}{x^2+j^2}\, dx
   &= \frac{1}{(2n+1)n!(n+1)!},\\
  \frac{2}{\pi} \int_{0}^{\infty} \prod_{j=1}^{n+1} \frac{1}{x^2+(2j-1)^2} \, dx
   &= \frac{1}{2^{2n}(2n+1)(n!)^2},\\
  \frac{2}{\pi} \int_{0}^{\infty} \prod_{j=1}^{n} \frac{1}{x^{2} + 1/j^{2}} \, dx
   &= \frac{2A(2n-1,n-1)}{\binom{2n}{n}}
\end{align*}
may be deduced inductively from Lemma \ref{Gnqpartialfractions}. Here, $A(n,k)$ are the Eulerian numbers which count
the number of permutations of $n$ objects with exactly $k$ descents. Recall that a permutation
$\sigma$ of the $n$ letters ${1,2,\ldots,n}$, here written as $\sigma(1)\,\sigma(2)\,\ldots\,\sigma(n)$,
has a descent at position $k$ if $\sigma(k)>\sigma(k+1)$. For
instance, $A(3,1)=4$ because there are $4$ permutations of $1,2,3$, namely $1\,3\,2$,
$2\,1\,3$, $2\,3\,1$, and $3\,1\,2$, which have exactly one
descent.
\end{example}

The problem of finding an explicit formula for the numerator appearing on
the right-hand side of (\ref{formula-0}) when put on the lowest common denominator is discussed in the next section.

\section{Representation in terms of Schur functions}\label{sec-schur}

The expression for $G_{n}(\boldsymbol{q})$ developed in
this section is given in terms of {\em Schur functions}. The reader is referred to \cite{bressoud1}
for a motivated introduction to these functions in the context of
{\em alternating sign matrices} and to \cite{sagan-book1} for their role in the
representation theory of the symmetric group. Among the many equivalent
definitions for Schur functions, we
now recall their definition in terms of quotients of alternants. Using this approach,
we are able to associate a Schur function not only to a partition
but more generally to an arbitrary vector.

Here, a vector $\mu = (\mu_1, \mu_2, \ldots)$ means a finite sequence of real
numbers. $\mu$ is further called a partition (of $m$) if $\mu_1 \ge \mu_2
\ge \cdots$ and all the parts $\mu_j$ are positive integers (summing up to $m$).
Write $\boldsymbol{1}^n$ for the partition with $n$ ones, and denote by
$\lambda(n)$ the partition
\[ \lambda(n) = (n-1, n-2, \ldots, 1). \]
Vectors and partitions may be added componentwise. In case they are of
different length, the shorter one is padded with zeroes. For instance, one has
$\lambda(n+1) = \lambda(n) + \boldsymbol{1}^n$. Likewise, vectors and
partitions may be multiplied by scalars. In particular, $a \cdot
\boldsymbol{1}^n$ is the partition with $n$ $a$'s.

Fix $n$ and consider $\boldsymbol{q}= (q_1, q_2, \ldots, q_n)$. Let $\mu =
(\mu_1, \mu_2, \ldots)$ be a vector of length at most $n$. The corresponding
alternant $a_{\mu}$ is defined as the determinant
\[ a_{\mu} (\boldsymbol{q}) = \left| q_i^{\mu_j} \right|_{1 \le i, j
   \le n} . \]
Again, $\mu$ is padded with zeroes if necessary. Note that the alternant
$a_{\lambda (n)}$ is the classical Vandermonde determinant
\[ a_{\lambda (n)} (\boldsymbol{q}) = \left| q_i^{n - j} \right|_{1 \le i,
   j \le n} = \prod_{1 \le i < j \le n} (q_i - q_j). \]
The Schur function $s_{\mu}$ associated with the vector $\mu$ can now be
defined as
\[ s_{\mu} (\boldsymbol{q}) = \frac{a_{\mu + \lambda(n)}
   (\boldsymbol{q})}{a_{\lambda(n)} (\boldsymbol{q})} . \]
If $\mu$ is a partition with integer entries this is a symmetric polynomial.
Indeed, as $\mu$ ranges over the partitions of $m$ of length at most $n$, the
Schur functions $s_{\mu} (\boldsymbol{q})$ form a basis for the homogeneous
symmetric polynomials in $\boldsymbol{q}$ of degree $m$.

The Schur functions include as special cases the {\em elementary symmetric
functions} $e_k$ and the {\em complete homogeneous symmetric functions} $h_k$.
Namely, $e_k(\boldsymbol{q}) = s_{\boldsymbol{1}^k}(\boldsymbol{q})$ and $h_k
(\boldsymbol{q}) = s_{(k)}(\boldsymbol{q})$.

The next result expresses the integral $G_n(\boldsymbol{q})$ defined in
(\ref{wallis-3}) as a quotient of Schur functions.

\begin{theorem}
\label{thm-gn}
Let $\boldsymbol{q} = (q_1, \ldots, q_n)$ and $\RealPart{q_k} > 0$. Then
\begin{equation}\label{gn-def}
G_{n}(\boldsymbol{q}) = \frac{s_{\lambda(n-1)}(\boldsymbol{q})}{s_{\lambda(n+1)}(\boldsymbol{q})} =
\frac{s_{\lambda(n-1)}(\boldsymbol{q})}{e_{n}(\boldsymbol{q}) s_{\lambda(n)}(\boldsymbol{q})}.
\end{equation}
\end{theorem}
\begin{proof}
The equality $e_{n}(\boldsymbol{q}) s_{\lambda(n)}(\boldsymbol{q})=
s_{\lambda(n+1)}(\boldsymbol{q})$ amounts to the identity
\[ q_{1}q_{2}\cdots q_{n} \left| q_{i}^{2n-2j} \right|_{i,j}
   = \left| q_{i}^{2n-2j+1} \right|_{i,j}, \]
which follows directly by inserting the factor $q_i$ into row $i$ of the matrix.

From the previous definition of Schur functions, the right-hand side
of (\ref{gn-def}) becomes
\[
  \frac{s_{\lambda(n-1)}(\boldsymbol{q}) }{e_n(\boldsymbol{q}) \, s_{\lambda(n)}(\boldsymbol{q}) }
  = \frac{a_{\lambda(n-1)+ \lambda(n)}(\boldsymbol{q})}{e_n(\boldsymbol{q}) a_{2 \lambda(n)}(\boldsymbol{q})}.
\]
Observe that $a_{2 \lambda(n)}(\boldsymbol{q}) = |q_{i}^{2n-2j}|_{i,j} = a_{\lambda(n)}(\boldsymbol{q}^2)$
is simply the Vandermonde determinant with $q_i$ replaced
by $q_i^2$. Next, expand the determinant $a_{\lambda(n-1) + \lambda(n)}$ by the
last column (which consists of $1$'s only) to find
\begin{equation*}
a_{\lambda(n-1)+\lambda(n)}(\boldsymbol{q}) = e_{n}(\boldsymbol{q}) \sum_{k=1}^n
 \frac{(-1)^{n-k}}{q_{k}}  a_{\lambda(n-1)}(q_{1}^{2}, q_{2}^{2}, \ldots,
q_{k-1}^{2}, q_{k+1}^{2}, \ldots, q_{n}^{2}).
\end{equation*}
Therefore
\begin{equation}\label{mess}
\frac{a_{\lambda(n-1)+\lambda(n)}(\boldsymbol{q})}{e_{n}(\boldsymbol{q}) \, a_{2 \lambda(n)}(\boldsymbol{q})} =
\sum_{k=1}^{n} \frac{(-1)^{n-k}}{q_{k}}
\mathop{\prod_{i<j}}_{i,j \neq k} (q_{i}^{2} - q_{j}^{2})
\Big{/}
\prod_{i<j} (q_{i}^{2} - q_{j}^{2}).
\end{equation}
Observe that the only terms that do not cancel in the
quotient above are those for which $i=k$ or $j=k$. The change of
sign required to transform the factors $q_{k}^{2}-q_{j}^{2}$ to
$q_{j}^{2}-q_{k}^{2}$ eliminates the factor $(-1)^{n-k}$. The
expression on the right-hand side of (\ref{mess}) is precisely
the value (\ref{formula-0}) of the integral $G_{n}(\boldsymbol{q})$ produced by
partial fractions.
\end{proof}

The next example illustrates Theorem \ref{thm-gn} with the principal
specialization of the parameters $\boldsymbol{q}$.

\begin{example}
The special case $q_k = q^k$ produces the evaluation
\begin{equation}\label{nice}
\frac{2}{\pi} \int_{0}^{\infty} \prod_{k=1}^n \frac{1}{x^2+q^{2k}}
= \frac{1}{q^{n^2}} \prod_{j=1}^{n-1} \frac{1-q^{2j-1}}{1-q^{2j}}.
\end{equation}
This can be obtained inductively from Lemma \ref{Gnqpartialfractions} but may also be
derived from Theorem \ref{thm-gn} in combination with the evaluation (\ref{eq-schurprincipal}) of the principal
specialization of Schur functions as in Theorem 7.21.2 of \cite{stanley-vol2}.

Taking the limit $q \to 1$ in (\ref{nice}) reproduces formula (\ref{qone}) for $G_n$.
In other words, (\ref{nice}) is a $q$-analog \cite{gasper1}
of (\ref{qone}). Similarly,
\[ \frac{\pi_q}{1 + q} = q^{1 / 4} \prod_{n=1}^\infty \frac{1 - q^{2 n}}{1
   - q^{2 n - 1}}  \frac{1 - q^{2 n}}{1 - q^{2 n + 1}} \]
is a useful $q$-analog of Wallis's formula (\ref{wallis2}) which naturally
appears in \cite{gosper1}, where Gosper studies $q$-analogs of trigonometric
functions (in fact, Gosper arrives at the above expression as a definition
for $\pi_q$ while $q$-generalizing the reflection formula $\Gamma(z)\Gamma(1-z) = \tfrac{\pi}{\sin \pi z}$).
\end{example}

The proof of Theorem \ref{thm-gn} extends to the following more general 
result.

\begin{lemma}\label{thm-gn-gen}
\[
  \sum_{k=1}^n \frac{1}{q_k^{\alpha-\beta}} \mathop{\prod_{j=1}^n}_{j \ne k} \frac{1}{q_j^\alpha - q_k^\alpha}
  = \frac{s_{\lambda}(\boldsymbol{q})}{s_{\mu}(\boldsymbol{q})},
\]
where
\begin{eqnarray*}
  \lambda & = & (\alpha-1) \cdot \lambda(n) - \beta \cdot \boldsymbol{1}^{n-1},\\
  \mu     & = & (\alpha-1) \cdot \lambda(n+1) - (\beta-1) \cdot \boldsymbol{1}^{n}.
\end{eqnarray*}
\end{lemma}

As a consequence, one obtains the following integral evaluation which generalizes the evaluation of
$G_{n}(\boldsymbol{q})$ given in Theorem \ref{thm-gn}.

\begin{theorem}\label{thm-gnt}
Let $\boldsymbol{q} = (q_1, \ldots, q_n)$ and $\RealPart{q_k} > 0$. Further,
let $\alpha$ and $\beta$ be given such that $\alpha > 0$, $0 < \beta < \alpha n$, and $\beta$ is not an integer
multiple of $\alpha$. Then
\[
  G_{n,\alpha,\beta}(\boldsymbol{q}) := \frac{\sin(\pi\beta/\alpha)}{\pi/\alpha}
  \int_0^\infty  \frac{x^{\beta-1}}{\prod_{k=1}^n (x^\alpha+q_k^\alpha)} \, dx
  = \frac{s_{\lambda}(\boldsymbol{q})}{s_{\mu}(\boldsymbol{q})},
\]
where $\lambda$ and $\mu$ are as in Lemma \ref{thm-gn-gen}.
\end{theorem}

\begin{proof}
Upon writing $\beta = b \alpha + \beta_1$ for $b<n$ a positive integer and $0<\beta_1<\alpha$,
the assertion follows from the partial fraction decomposition
\[
  \frac{x^{b \alpha}}{\prod_{k=1}^{n} (x^\alpha+ q_{k}^\alpha)}
  = (-1)^b \sum_{k=1}^{n} \frac{q_k^{b \alpha}}{x^\alpha+q_{k}^\alpha} \prod_{j \neq k}
  \frac{1}{q_{j}^\alpha - q_{k}^\alpha},
\]
the integral evaluation
\[ \int_{0}^{\infty} \frac{x^{\beta_1-1} dx}{x^\alpha + q^\alpha}
  = \frac{1}{q^{\alpha-\beta_1}} \, \frac{\pi/\alpha}{\sin(\pi\beta_1/\alpha)}, \]
and Lemma \ref{thm-gn-gen}.
\end{proof}

\section{Schur functions in terms of SSYT}\label{sec-ssyt}

The Schur function $s_{\lambda}(\boldsymbol{q})$ associated to a
partition $\lambda$
also admits
a representation in terms of {\em semi-standard Young tableaux} (SSYT). The
reader will find
information about this topic in \cite{bressoud1}. Given a partition
 $\lambda = (\lambda_1, \lambda_2, \ldots, \lambda_m)$, the Young diagram of shape $\lambda$
is an array of boxes, arranged in left-justified rows, consisting of $\lambda_{1}$ boxes in the first row, $\lambda_{2}$ in the second row,
and so on, ending with $\lambda_{m}$ boxes in the $m$th row.
A SSYT of shape $\lambda$ is a Young diagram of shape $\lambda$ in which
the boxes have been filled with positive integers.
These integers are restricted to be
weakly increasing across rows (repetitions are allowed) and strictly increasing
down columns.  From this point of view, the Schur function
$s_{\lambda}(\boldsymbol{q}) = s_{\lambda}(q_1, \ldots, q_n)$
can be defined as
\[ s_{\lambda}(\boldsymbol{q}) = \sum_{T} {\boldsymbol{q}}^{T}, \]
where the sum is over all SSYT of shape $\lambda$ with entries from
$\{1, 2, \ldots, n \}$. The symbol $\boldsymbol{q}^{T}$ is a monomial
in the variables $q_{j}$ in which the exponent of $q_{j}$ is the number
of appearances of $j$ in $T$. For example, the array shown in Figure
\ref{09-0469Fig1} is a tableau $T$ for the partition $(6, 4, 3, 3 )$. The
corresponding monomial $\boldsymbol{q}^{T}$
is given by $ q_{1}q_{2}^{3}q_{3}q_{4}^3 q_{5}^{4}q_{6}^{2}q_{7}q_{8}$.

\begin{figure}[h]
\begin{center}
\young(122455,2345,466,578)
\end{center}
\caption{A tableau $T$ for the partition $(6,4,3,3)$.}
\label{09-0469Fig1}
\end{figure}

The number $N_n(\lambda)$ of SSYT of shape $\lambda$ with entries from $\{1,2,\ldots,n\}$ can be
obtained by letting $q \to 1$ in the formula
\begin{equation}\label{eq-schurprincipal}
  s_{\lambda}(1,q,q^2, \ldots, q^{n-1}) = \prod_{1 \leq i<j \leq n} \frac{q^{\lambda_{i}+n-i}-q^{\lambda_{j}+n-j}}{q^{j-1}-q^{i-1}}
\end{equation}
(see page $375$ of \cite{stanley-vol2}). This yields
\begin{equation}\label{eq-nrofssyt}
N_n(\lambda) = \prod_{1 \leq i < j \leq n} \frac{\lambda_{i}-\lambda_{j}+j-i}{j-i}.
\end{equation}

The evaluation (\ref{qone}) of Wallis's integral (\ref{wallis1}) may be recovered
from here as
\[
  G_{n+1} = \frac{s_{\lambda(n)}(\boldsymbol{1}^{n+1})}
  {s_{\lambda(n+2)}(\boldsymbol{1}^{n+1})} = \frac{N_{n+1}(\lambda(n))}{N_{n+1}(\lambda(n+2))}
  = \frac{1} {2^{2n}} \binom{2n}{n}.
\]

\section{A counting problem}\label{sec-countingproblem}

The $k$-central binomial coefficients $c(n,k)$, defined
by the generating function
\[ (1-k^{2}x)^{-1/k} = \sum_{n \geq 0} c(n,k) x^{n}, \]
are given by
\[ c(n,k) = \frac{k^{n}}{n!} \prod_{m=1}^{n-1} (1 + km). \]
For $k=2$ these coefficients reduce to the central binomial coefficients $\tbinom{2n}{n}$.
The numbers $c(n,k)$ are integers in general and their divisibility properties have been
studied in \cite{ams-2}. In particular, the authors establish that the
$k$-central binomial coefficients are always divisible by $k$ and characterize
their $p$-adic valuations.

The next result attempts an (admittedly somewhat contrived) interpretation of what the numbers $-c(n,-k)$ count.

\begin{corollary}
Let $\lambda$ and $\mu$ be the partitions given by
\begin{eqnarray*}
  \lambda & = & (k-1) \cdot \lambda(n) - \boldsymbol{1}^{n-1},\\
  \mu     & = & (k-1) \cdot \lambda(n+1).
\end{eqnarray*}
Then the integer $-c(n,-k)$ enumerates the ratio between the total number of
SSYT of shapes $\lambda$ and $\mu$ with entries from $\{1,2,\ldots,n\}$ times the factor $k^{2n-1}/n$.
\end{corollary}

\begin{proof}
By Theorem \ref{thm-gnt},
\[
  \frac{N_n(\lambda)}{N_n(\mu)} = \frac{s_{\lambda}(\boldsymbol{1}^{n})}{s_{\mu}(\boldsymbol{1}^n)} = G_{n,k,1}(\boldsymbol{1}^n) = \frac{\sin(\pi/k)}{\pi/k} \int_0^\infty  \frac{1}{(x^k + 1)^n} \, dx,
\]
and the expression on the right-hand side is routine to evaluate:
\[ \frac{N_n(\lambda)}{N_n(\mu)} = \frac{\Gamma(n-\tfrac{1}{k})}{\Gamma(n) \Gamma(1-\tfrac{1}{k})} = \prod_{m=1}^{n-1} \frac{km-1}{km}. \]
This product equals $-c(n,-k)$ divided by $k^{2n-1}/n$.
\end{proof}

\begin{remark}
R. Stanley pointed out some interesting Schur function quotient results. See 
exercises $7.30$ and $7.32$ in \cite{stanley-vol2}.
\end{remark}

\section{An integral from Gradshteyn and Ryzhik}\label{final-gr}

It is now demonstrated how the previous results may be used to prove an
integral evaluation found as entry $3.112$ in \cite{gr}. The main tool is
the (dual) Jacobi-Trudi identity which expresses a Schur function in terms of elementary
symmetric functions. Namely, if $\lambda$ is a partition such that its
conjugate $\lambda'$ (the unique partition whose Young diagram, see Section \ref{sec-ssyt},
is obtained from the one of $\lambda$ by interchanging rows and columns) has length at most $m$ then
\[ s_{\lambda} = \left| e_{\lambda'_i - i + j} \right|_{1 \le i, j \le m}. \]
This identity may be found for instance in \cite[Corollary 7.16.2]{stanley-vol2}.
\begin{theorem}
  Let $f_n$ and $g_n$ be polynomials of the form
  \begin{eqnarray*}
    g_n (x) & = & b_0 x^{2 n - 2} + b_1 x^{2 n - 4} + \ldots + b_{n - 1},\\
    f_n (x) & = & a_0 x^n + a_1 x^{n - 1} + \ldots + a_n,
  \end{eqnarray*}
  and assume that all roots of $f_n$ lie in the upper half-plane. Then
  \[ \int_{- \infty}^{\infty} \frac{g_n (x) dx}{f_n (x) f_n (- x)}
     = \frac{\pi i}{a_0}  \frac{M_n}{\Delta_n}, \]
  where
  \[ \Delta_n = \left| \begin{array}{ccccc}
       a_1 & a_3 & a_5 & \ldots & 0\\
       a_0 & a_2 & a_4 &  & 0\\
       0 & a_1 & a_3 &  & 0\\
       \vdots &  &  & \ddots & \\
       0 & 0 & 0 &  & a_n
     \end{array} \right|, \hspace{2em} M_n = \left| \begin{array}{ccccc}
       b_0 & b_1 & b_2 & \ldots & b_{n - 1}\\
       a_0 & a_2 & a_4 &  & 0\\
       0 & a_1 & a_3 &  & 0\\
       \vdots &  &  & \ddots & \\
       0 & 0 & 0 &  & a_n
     \end{array} \right| . \]
\end{theorem}

\begin{proof}
  Write $f_n (x) = a_0 \prod_{j = 1}^n (x - i q_j)$. By assumption,
  $\RealPart{q_j} > 0$. Further,
  \[ f_n (x) f_n (- x) = (- 1)^n a_0^2 \prod_{j = 1}^n (x^2 + q_j^2). \]
  Let $\boldsymbol{q}= (q_1, q_2, \ldots, q_n)$. It follows from Theorem \ref{thm-gnt} that
  \[ \int_{- \infty}^{\infty} \frac{x^{2 \beta} dx}{f_n (x) f_n (- x)}
     = \frac{(- 1)^{n + \beta} \pi}{a_0^2}  \frac{s_{\lambda (n - 1)
     - 2 \beta \cdot \boldsymbol{1}^{n - 1}} (\boldsymbol{q})}{s_{\lambda (n + 1)
     - 2 \beta \cdot \boldsymbol{1}^n} (\boldsymbol{q})} = \frac{(- 1)^n
     \pi}{a_0^2}  \frac{s_{\lambda'} (\boldsymbol{q})}{s_{\lambda (n + 1)}
     (\boldsymbol{q})} \]
  where $\lambda = \lambda (n - 1) + 2 \cdot \boldsymbol{1}^{\beta}$. The latter
  equality is obtained by writing the quotient of Schur
  functions as a quotient of alternants, multiplying the $k$th row of each matrix by
  $q_k^{2 \beta}$, and reordering the columns of the determinant in the
  numerator. The right-hand side now is a quotient of Schur functions to which
  the Jacobi-Trudi identity may be applied. In the denominator, this gives
  \[ s_{\lambda (n + 1)} (\boldsymbol{q}) = \left| e_{n + 1 - 2 k + j}
     (\boldsymbol{q}) \right|_{1 \le k, j \le n} = \left| e_{2 k -
     j} (\boldsymbol{q}) \right|_{1 \le k, j \le n} . \]
  Note that $e_m(\boldsymbol{q}) = 0$ whenever $m<0$ or $m>n$. Further,
  $e_k (\boldsymbol{q}) = i^k a_k / a_0$. Hence, $s_{\lambda (n + 1)}
  (\boldsymbol{q}) = i^{n (n + 1) / 2} \Delta_n / a_0^n$. The term $s_{\lambda'}
  (\boldsymbol{q})$ is dealt with analogously. The claim follows by expanding
  the determinant $M_n$ with respect to the first row.
\end{proof}

\section{A sum related to Feynman diagrams}\label{sec-feynman}

Particle scattering  in quantum field theory is usually described in terms
of Feynman diagrams. A Feynman diagram
is a graphical representation of a particular
term arising in the expansion of the relevant quantum-mechanical
scattering amplitude as a power series in the coupling constants that
parametrize the strengths of the interactions.

From the mathematical point of view, a Feynman diagram is a graph to
which a certain function is associated. If the graph has circuits
(\emph{loops}, in the physics terminology) then this function is
defined in terms of a number of integrals over the $4$-dimensional
momentum space $(k_{0}, \boldsymbol{k})$, where $k_{0}$ is the \emph{energy}
integration variable and $\boldsymbol{k}$ is a $3$-dimensional \emph{momentum}
variable.

Feynman diagrams also appear in calculations of the thermodynamic
properties of a system described by quantum fields. In this context, the
integral over the
energy component of a Feynman loop diagram is replaced by a summation over
discrete energy values. These Matsubara sums were introduced in
\cite{matsubara-1}.  A general method to compute these sums
in terms of an associated integral was presented in \cite{espinosa-5}.

These techniques, applied to the expression
 (\ref{formula-0}) for the integral
  $G_n(\boldsymbol{q})$, give the value of the sum associated with
 the one-loop Feynman diagram consisting of $n$ vertices and vanishing external momenta,
  $N_i=0$, as depicted in Figure \ref{09-0469Fig2}.

{{
\begin{figure}[h]
\begin{center}
\centerline{\includegraphics{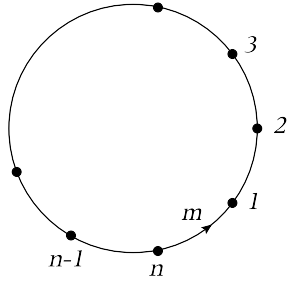}}
\caption{The one-loop Feynman diagram with $n$ vertices and vanishing external momenta. $m$ is
the summation variable associated to each of the internal lines.}
\label{09-0469Fig2}
\end{center}
\end{figure}
}}

The Matsubara sum associated to the diagram in Figure \ref{09-0469Fig2} is
\begin{equation}\label{matsubarasum}
M_n(\boldsymbol{q}):=\sum_{m=-\infty}^{\infty} \prod_{k=1}^n \frac{1}{m^2+q_k^2},
\end{equation}
where the variables $q_k$ are related to the kinematic energies carried by the
(virtual) particles in the Feynman diagram.
This sum was denoted by $S_{G}$ in \cite{espinosa-5}; the notation has been changed
here to avoid confusion.

\begin{example}
The first few Matsubara sums are
\begin{align*}
M_1(q_1)&=\pi\frac{D_1}{q_1},\\
M_2(q_1,q_2)&=\pi\frac{q_2 D_1 - q_1 D_2}{q_1 q_2 (q_2^2 - q_1^2)},\\
M_3(q_1,q_2,q_3)&=\pi\frac{q_2 q_3 (q_2^2-q_3^2)D_1 + q_3 q_1 (q_3^2-q_1^2)D_2 + q_1 q_2 (q_1^2-q_2^2)D_3}{q_1 q_2 q_3 (q_3^2-q_2^2)(q_2^2-q_1^2)(q_1^2-q_3^2)},
\end{align*}
with $D_j=\coth(\pi q_j)$.
\end{example}

\begin{theorem}
The Matsubara sum $M_n(\boldsymbol{q})$ is given by
\[ M_n(\boldsymbol{q}) = \pi \sum_{k=1}^n \frac{\coth(\pi q_k)}{q_k}
   \mathop{\prod_{j=1}^n}_{ j\neq k} \frac{1}{q_j^2 - q_k^2}. \]
\end{theorem}

\begin{proof}[Proof 1]
This follows from the partial fraction expansion
\[ \prod_{k=1}^n \frac{1}{m^2 + q_k^2} = \sum_{k=1}^n \frac{1}{q_k^2 + m^2}
   \prod_{j \neq k} \frac{1}{q_j^2 - q_k^2}, \]
which is a special case of (\ref{pf-1}), switching the order of summation, and employing the classical
\[ \frac{\pi \coth(\pi z)}{z} = \sum_{m=-\infty}^\infty \frac{1}{z^2+m^2}. \]
\end{proof}

\begin{proof}[Proof 2]
The method developed in \cite{espinosa-5} shows that
\begin{equation}\label{matsu-00}
M_n(\boldsymbol{q})=\pi \left[ 1 + \sum_{m=1}^n \nbe(q_m)(1-R_m) \right] G_n(\boldsymbol{q}),
\end{equation}
where $G_n(\boldsymbol{q})$ is the integral defined in (\ref{wallis-3}),
\[ \nbe(q) = \frac{1}{e^{2\pi q}-1} = \frac{1}{2}\left(\coth \pi q - 1\right), \]
and $R_m$ is the reflection operator defined by
\[ R_m f(q_1,\ldots,q_m,\ldots)=f(q_1,\ldots,-q_m,\ldots). \]
To use (\ref{matsu-00}) combined with the evaluation (\ref{formula-0}) of $G_n(\boldsymbol{q})$
it is required to compute the action of each $1-R_m$ on the summands of (\ref{formula-0}). Namely,
\begin{equation*}
(1-R_m) \frac{1}{q_k} \mathop{\prod_{j=1}^n}_{j\neq k} \frac{1}{q_j^2 - q_k^2}
= \frac{2 \delta_{km}}{q_k} \mathop{\prod_{j=1}^n}_{j\neq k}  \frac{1}{q_j^2 - q_k^2},
\end{equation*}
where $\delta_{km}$ is the Kronecker delta. Therefore
\begin{equation*}
\nbe(q_m)(1-R_m)G_n(\boldsymbol{q})
= \frac{2 \nbe(q_m)}{q_m} \mathop{\prod_{j=1}^n}_{j\neq m} \frac{1}{q_j^2 - q_m^2},
\end{equation*}
and the result follows from $2\nbe(q)=\coth(\pi q)-1$.
\end{proof}

Finally, an expansion of $M_n(\boldsymbol{q})$ in terms of symmetric functions is given.
Starting with the classical expansion
\[ \frac{\pi \coth q_k}{q_k} = \frac{1}{q_k^2} - 2\sum_{m=1}^\infty (-1)^m q_k^{2m-2} \zeta(2m), \]
where $\zeta(s)$ denotes the Riemann zeta function, it follows that
\[ M_{n} ( \boldsymbol{q}) = \sum_{k=1}^{n} \frac{1}{q_{k}^{2}}
\prod_{j \neq k} \frac{1}{q_{j}^{2}-q_{k}^{2}} - 2 \sum_{m=1}^{\infty} (-1)^{m} \zeta(2m)
\sum_{k=1}^{n} q_{k}^{2(m-1)} \prod_{j \neq k} \frac{1}{q_{j}^{2} - q_{k}^{2}}. \]
Using the identity ($h_{j}$ being the complete homogeneous symmetric 
function)
\[ h_{m-n}(x_{1},\ldots, x_{n}) = (-1)^{n-1} \sum_{k=1}^{n} x_{k}^{m-1} \prod_{j \neq k} \frac{1}{x_{j}-x_{k}}, \]
which follows from Lemma \ref{thm-gn-gen} (or see page $450$, Exercise $7.4$ of
\cite{stanley-vol2}), this proves:

\begin{corollary}
The Matsubara sum $M_n(\boldsymbol{q})$, defined in (\ref{matsubarasum}), is given by
\[ M_{n}(\boldsymbol{q}) = \frac{1}{e_{n}(\boldsymbol{q}^2)} +
   2 \sum_{m=0}^{\infty} (-1)^{m} \zeta(2m+2n) h_{m}(\boldsymbol{q}^{2}). \]
\end{corollary}

\section{Conclusions}\label{sec-conclusions}

The evaluation of definite integrals has the charming quality of taking the reader
for a tour of many parts of mathematics. An innocent-looking generalization of 
one of the oldest formulas in analysis has been shown to connect the work of 
the four authors in the title. 

\medskip

\noindent
{\bf Acknowledgements.} The authors 
wish to thank B.~Berndt for pointing out the 
results described in Section \ref{final-gr}. The second author would
like to thank
the Department of Mathematics at Tulane University for its hospitality.
The third author acknowledges the partial support of
NSF-DMS 0713836. The last author 
is a graduate student partially supported by the
same grant.

\bigskip

\noindent\textbf{Tewodros Amdeberhan} received his Ph.D. in mathematics from Temple 
University under 
the supervision of Doron Zeilberger. His main interest lies in WZ-theory, 
number theory, and elliptic PDEs. He has been a visiting scholar at Princeton and MIT. 
Currently, he teaches at Tulane University and holds a permanent membership 
at DIMACS, Rutgers University. In his spare time he enjoys playing chess and 
soccer.

\noindent\textit{Mathematics Department, 
Tulane University, New Orleans, LA 70118\\
tamdeberhan@math.tulane.edu}

\bigskip

\noindent\textbf{Olivier R.~Espinosa} was born in Valparaiso, Chile. He received a Ph.D. in
particle physics from Caltech in 1990, and joined the faculty of Universidad
Tecnica Federico Santa Maria in Valparaiso in 1992, where he is now professor
of physics. Although his main area of research has been quantum field theory,
since 2000 he has also mantained a fruitful collaboration with Victor~H.
Moll of Tulane University on the pure mathematics of some integrals and
infinite sums connected with the Hurwitz zeta function. Olivier is married
to Nina and is the father of two college-aged children.

\noindent\textit{Departamento de F\'{\i}sica,
Universidad T\'{e}cnica Federico Santa Mar\'{\i}a, Casilla 110-V, Valpara\'{\i}so, Chile\\
olivier.espinosa@usm.cl}

\bigskip

\noindent\textbf{Victor H.~Moll} is a professor of mathematics at Tulane University. He 
studied under Henry McKean at the Courant Institute of New York University. 
In the last decade of the last century, through the influence of a graduate 
student (George Boros), he entered the world of \texttt{Integrals}.
Before coming to Tulane, he spent two years at Temple 
University in Philadelphia, where he was unknowingly coached  by 
Donald J. Newman. He enjoys collaborating with students and 
anyone interested in the large variety of topics related to the evaluation 
of integrals. In his spare time, he enjoys the wonderful things that New 
Orleans has to offer: great music and fantastic food. 

\noindent\textit{Mathematics Department, 
Tulane University, New Orleans, LA 70118\\
vhm@math.tulane.edu}

\bigskip

\noindent\textbf{Armin Straub} received one of the last classic diplomas from Technische
 Universit\"at Darmstadt, Germany, in 2007 under the guidance of Ralf
 Gramlich. At the moment, he is pursuing his Ph.D. at Tulane University, New
 Orleans, where his appetite for combinatorics, special functions, and
 computer algebra is constantly nurtured by his advisor Victor Moll.
 Besides mathematics of almost all sorts he is excited about currently
 teaching his first course, and particularly enjoys eating and playing
 soccer.

\noindent\textit{Mathematics Department, 
Tulane University, New Orleans, LA 70118\\
astraub@math.tulane.edu}

\end{document}